\documentclass[preprint,12pt]{elsarticle}
\usepackage{tikz}
\usetikzlibrary{positioning, shapes.geometric}
\usepackage{amssymb}
\usepackage{amsmath}

\numberwithin{equation}{section}
\newtheorem{theorem}{Theorem}[section]
\newtheorem{corollary}[theorem]{Corollary}

\newtheorem{lemma}[theorem]{Lemma}
\newtheorem{proposition}[theorem]{Proposition}

\newtheorem{example}[theorem]{Example}

\newtheorem{definition}[theorem]{Definition}

\newproof{proof}{Proof}

\journal{.}

\begin{document}

\begin{frontmatter}

\title{A class of core inverses associated with Green's relations in semigroups}

\author[1]{Huihui Zhu\corref{cor}}
\ead{hhzhu@hfut.edu.cn}
\cortext[cor]{Corresponding author}

\author[1]{Bing Dong}
\ead{bdmath@163.com}

\address[1]{School of Mathematics, Hefei University of Technology, Hefei 230009, China.}

\begin{abstract}
Let $S$ be a $*$-monoid and let $a,b,c$ be elements of $S$. We say that $a$ is $(b,c)$-core-EP invertible if there exist some $x$ in $S$ and some nonnegative integer $k$ such that $cax(ca)^{k}c=(ca)^{k}c$, $x{\mathcal R}(ca)^{k}b$ and $x{\mathcal L}((ca)^{k}c)^{*}$. This terminology can be seen as an extension of the $w$-core-EP inverse and the $(b,c)$-core inverse. It is explored when $(b,c)$-core-EP invertibility implies $w$-core-EP invertibility. Another accomplishment of our work is to establish the criteria for the $(b,c)$-core-EP inverse of $a$ and to clarify the relations between the $(b,c)$-inverse, the core inverse, the core-EP inverse, the $w$-core inverse, the $(b,c)$-core inverse and the $(b,c)$-core-EP inverse. As an application, we improve a result in the literature focused on $(b,c)$-core inverses. We then establish the criterion for the $(B,C)$-core-EP inverse of $A$ in complex matrices, and give the solution to the system of matrix equations.
\end{abstract}


\end{frontmatter}

\section{History and motivation}

A framework of the theory of generalized inverses was developed in 1903 by Fredholm in his study on integral equations \cite{Fredholm1903}. In 1906, Moore \cite{Moore1920} introduced the ``reciprocal inverse'' of a matrix. This work was published in 1920. In 1955, Penrose \cite{Penrose1955} showed that the ``reciprocal inverse'' in the sense of Moore can be expressed by a system of four equations. Over the years, several kinds of generalized inverses in different algebraic systems were introduced, for instance the Drazin inverse of an associative ring element as well as a semigroup element \cite{Drazin1958}, the generalized Moore--Penrose inverse of matrices over an integer domain \cite{Prasad1992}, the generalized Drazin inverse of an element in Banach algebra \cite{Koliha1996}. An explosive development in the theory of generalized inverses came in the 1950s with a wide range of applications, for instance networks, linear estimation, Markov chains, coding theory and cryptography, see \cite{Ben1976, BD1953,Dawson1997,Elliott2006,Hunter2014, Marquaridt1970, Meyer1975, Sun2001, Wu1998} and the references therein; for their subtle connections with clean decompositions of ring elements, we refer to, for example, \cite{Lam2014, Lam2019, Zhu2019}.

During the first half of the 2010s, a couple of types of ``mixed'' generalized inverses in complex matrices have been introduced, such as the core inverse \cite{Baksalary2010}, the core-EP inverse \cite{Prasad2014} (named pseudo core inverse \cite{Gao2013} in a $*$-ring) and the DMP inverse \cite{Malik2014}. The definitions can be viewed as a combination of the classical generalized inverses (precisely, the group inverse and the Drazin inverse) and the Moore--Penrose inverse.

To explore what is the mixed version of new generalized inverses (Mary's inverse along an element \cite{Mary2011} and Drazin's $(b,c)$-inverse \cite{Drazin2012}) and a $\{1,3\}$-inverse, the $w$-core inverse \cite{Zhu2023} was introduced in a $*$-monoid. It is exactly both invertible along an element and $\{1,3\}$-invertible. Recently Zhu \cite{Zhu2024} introduced a more general $(b,c)$-core inverse in a $*$-monoid, see \cite{Zhu20231,Zhu20232, Zhu20241} for a detailed description. Subsequently, another generalization of the $w$-core inverse, known as the $w$-core-EP inverse, arose in the work of Mosi\'{c} et al. \cite{Mosic2023} in a $*$-ring. The three new generalized inverses extend the core inverse, the core-EP inverse and the DMP inverse in the ring of complex matrices.

This motivates us to consider the question whether we can give a unified theory to encompass the $w$-core-EP inverse and the $(b,c)$-core inverse. Precisely, the motivation of the paper comes from the following incomplete diagram of the related notions:

\begin{center}
\begin{tikzpicture}
[node distance=30pt] \node[draw,text width=8.5em, text height=0.8em, text centered] (core) {core inverses}; \node[draw,text width=8.5em, text height=0.8em, text centered, below=of core] (w) {$w$-core inverses}; \node[draw,text width=8.5em, text height=0.8em, text centered, right=30pt of core] (core-EP) {core-EP inverses}; \node[draw,text width=8.5em, text height=0.8em, text centered, right=30pt of core-EP] (w-EP) {$w$-core-EP inverses}; \node[draw,text width=8.5em, text height=0.8em, text centered, below=of core-EP] (bc-core) {$(b,c)$-core inverses};\node[draw,text width=8.5em, text height=0.8em, text centered, below=of w-EP] (?) {?}; \draw[->,thick] (core) -- (w); \draw[->,thick] (core) -- (core-EP); \draw[->,thick] (core-EP) -- (w-EP); \draw[->,thick] (w) -- (bc-core); \draw[->,thick] (core-EP) -- (bc-core);\draw[->,thick] (w-EP) -- (?); \draw[->,thick] (bc-core) -- (?);
\end{tikzpicture}
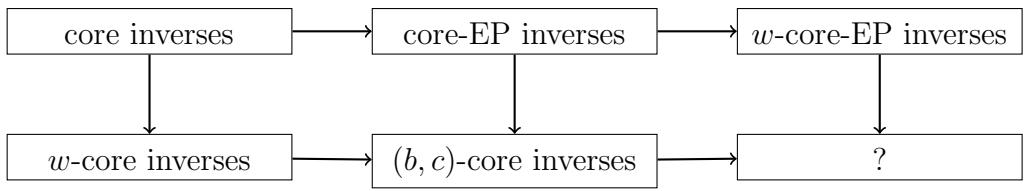
\begin{center} Figure 1. Implications between several types of core inverses. \end{center}
\end{center}

The main purpose of this paper is to introduce and investigate the $(b,c)$-core-EP inverse (see Definition \ref{3.1} below) in order to complete Figure 1. All implications in Figure 1 follow from \cite{Prasad2014,Zhu2024,Zhu2023}. When the above diagram is done, it will be helpful in enriching the theory of generalized inverses may be appropriate to given potential applications in the mathematical literature.

\section{Preliminaries}

Let us first recall some necessary background material. Let the symbol $\mathbb{N}=\{0,1,2,3, \cdots\}$ be the set of all natural numbers, and let $\mathbb{N}^{*}=\{1,2,3, \cdots\}$ be the set of all natural numbers without zero. Throughout this paper, we'll use the widely accepted shorthand ``iff'' for ``if and only if'' in the text. Other standard terminology and notation in semigroup and ring theory follow those in \cite{Anderson1974,Howie1955}.

Let $S$ be a semigroup. We define $S^1 = \left\{
    \begin{aligned}
    &S & & {\rm ~if ~S ~is~ a~ monoid,} \cr
    &  S \cup \{1\}& &{ \rm ~otherwise.}
    \end{aligned}
\right.$
Green's relations and Green's preorders were first defined and studied in \cite{Green1951}. They are an important tool in analysing the structure of semigroups, see, e.g., \cite{Brookes2024, Mary2023}. The three Green's relations ${\mathcal L}$, ${\mathcal R}$ and ${\mathcal H}$, and the three Green's preorders $\leq_{\mathcal L}$, $\leq_{\mathcal R}$, $\leq_{\mathcal H}$ are defined as follows:

(1) $a\leq_{\mathcal L}b$ iff $S^1a \subseteq S^1b$, iff $a=xb$ for some $x\in S^1$.

(2) $a\leq_{\mathcal R}b$ iff $aS^1\subseteq bS^1$, iff $a=by$ for some $y\in S^1$.

(3) $a\leq_{\mathcal H}b$ iff $a\leq_{\mathcal L} b$ and $a\leq_{\mathcal R}b$.

(4) $a\mathcal{L}b$ iff $S^1a=S^1b$, iff $a=xb$ and $b=ya$ for some $x,y\in S^1$.

(5) $a\mathcal{R} b$ iff $aS^1=bS^1$, iff $a=bx$ and $b=ay$ for some $x,y\in S^1$.

(6) $a\mathcal{H} b$ iff $a{\mathcal L}b$ and $a{\mathcal R}b$.

Based on Green's preorders, the well known inverse along an element was first explicitly introduced by Mary in his celebrated paper \cite{Mary2011}. For any $a,d\in S$, the element $a$ is invertible along $d$ if there exists some $b\in S$ such that $bad=d=dab$ and $b\leq_{\mathcal H} d$. Such an element $b$ is called the inverse of $a$ along $d$. It is unique if it exists, and is denoted by $a^{\parallel d}$.

In his influential work \cite{Drazin2012}, Drazin introduced the $(b,c)$-inverse in a semigroup. Given any $a,b,c\in S$, $a$ is said to be $(b, c)$-invertible if there exists $y\in S$ such that $y\in bSy\cap ySc$, $yab=b$ and $cay=c$. Such an $y$ is called the $(b, c)$-inverse of $a$. It is unique when it exists, and is denoted by $a^{(b,c)}$. We denote by $S^{(b,c)}$ the set of all $(b,c)$-invertible elements in $S$. In the article \cite{Drazin2016} he proved, among others, a theorem which states that $a$ is $(b, c)$-invertible iff $a$ is both left $(b,c)$-invertible (equivalently, $b\in Scab$) and right $(b,c)$-invertible (equivalently, $c \in cabS$). In particular, the $(b,b)$-inverse of $a$ is exactly the inverse of $a$ along $b$. Both the inverse along an element and the $(b,c)$-inverse have become a staple of generalized inverses due to their inclusiveness and connection to the regularity of elements, for example, for any $a\in S$, an element $a$ is group invertible (a.k.a. strongly regular) iff $a$ is $(a,a)$-invertible, iff $a$ is invertible along $a$; $a$ is Drazin invertible (a.k.a. strongly $\pi$-regular) iff $a$ is $(a^n,a^n)$-invertible, iff $a$ is invertible along $a^n$ for some $n\in \mathbb{N^*}$; when $S$ is a $*$-monoid, $a$ is Moore--Penrose invertible (a.k.a. $*$-regular) iff $a$ is left (resp., right) $(a^*,a^*)$-invertible iff $a$ is left (resp., right) invertible along $a^*$. These equivalences can be found in references \cite{Drazin2012,Mary2011,Zhu2016}.

We assume that $R$ is a unital $*$-ring, that is a unital ring $R$ with an involution $*:R \rightarrow R$ satisfying $(a^*)^*=a$, $(ab)^{*}=b^{*}a^{*}$ and $(a+b)^*=a^*+b^*$ for all $a,b\in R$.

Originally introduced by Baksalary and Trenkler \cite{Baksalary2010} for complex matrices in 2010, and then extended to a $*$-ring $R$ (cf. \cite{Rakic2014}),  core inverse plays an important part in the theory of generalized inverses. The research of core inverses blossomed into a popular topic in the theory of generalized inverses ever since it was defined. In 2014, Manjunatha Prasad and Mohana \cite{Prasad2014} introduced the core-EP inverse of a complex matrix, extending the core inverse. Research interests in related problems of the core-EP inverse have been established in \cite{Dolinar2019, Kyrchei2021, Ma2018, Mosic20201}.

Recently in \cite{Mosic2023}, Mosi\'{c} et al. introduced the $w$-core-EP inverse in $R$. Let $a,w\in R$. The element $a$ is called $w$-core-EP invertible if there exists some $x\in R$ satisfying $awx^{2}=x$, $x(aw)^{k+1}a=(aw)^ka$ and $(awx)^{*}=awx$ for some $k\in \mathbb{N}$. Such an $x$, usually denoted by $a_{w}^{\tiny{\textcircled{D}}}$, is unique if it exists, and is called the $w$-core-EP inverse of $a$. By $R_{w}^{\tiny{\textcircled{D}}}$ we denote the set of all $w$-core-EP invertible elements in $R$. The smallest nonnegative integer $k$ is called the $w$-core-EP index of $a$, and is denoted by ${\rm i}_{w}(a)$. It follows from \cite{Mosic2023} that when $k=0$, the $w$-core-EP inverse is reduced to the $w$-core inverse, the $w$-core-EP inverse is simplified to the pseudo core inverse provided that $w=1$. It turns out that the pseudo core inverse is exactly the core-EP inverse in the ring of complex matrices.

Throughout, unless otherwise indicated, $S$ denotes a $*$-monoid, that is a semigroup $S$ which contains the unity $1$ and an involution $*: S\rightarrow S$ satisfying $(a^*)^*=a$ and $(ab)^{*}=b^{*}a^{*}$  for all $a,b\in S$.

Following the terminology in \cite[Definition 2.1]{Zhu2024}, for any $a,b,c\in S$, the element $a$ is called $(b,c)$-core invertible if there exists some $x\in S$ such that $caxc=c$, $xS=bS$ and $Sx=Sc^{*}$. Such an $x$ is called a $(b,c)$-core inverse of $a$. It is unique if it exists and is denoted by $a_{(b,c)}^{\tiny{\textcircled{\#}}}$. Examples include all $w$-core invertible elements, $w$-core-EP invertible elements and all Moore--Penrose invertible elements \cite{Penrose1955}.

Here is an outline of the paper. In Section 1, the origin and development of the theory of generalized inverses and the motivation of the paper are given. In Section 2 we recall some useful notations and definitions. In Section 3, for any $a,b,c\in S$, we formulate and prove that the $(b,c)$-core-EP inverse of $a$ is the uniquely determined $x\in S$ satisfying $cax(ca)^{k}c=(ca)^{k}c$, $x{\mathcal R}(ca)^{k}b$ and $x{\mathcal L}((ca)^{k}c)^{*}$ for some $k\in \mathbb{N}$. We give some criteria, basic properties as well as its formulae in terms of a class of extended Green's preorders. Also, the connections between the extended Green's preorders and the minus partial order are given. We then show that there are only two cases for the cardinality of ${\rm I}_{(b,c)}(a)$, the set of all $k\in \mathbb{N}$ satisfying the definition above. Precisely, it is proved that ${\rm card}({\rm I}_{(b,c)}(a))=1$ or ${\rm card}({\rm I}_{(b,c)}(a))=\infty$ whenever ${\rm card}({\rm I}_{(b,c)}(a))>1$. Particular emphasis is given to the connection between the $(b,c)$-core-EP inverse and the $w$-core-EP inverse. It is concluded that the $(b,c)$-core-EP invertibility of $a$ concludes the $a$-core-EP invertibility of $c$ provided that ${\rm card}({\rm I}_{(b,c)}(a))>1$. Moreover, the $(b,c)$-core-EP index of $a$ is no more than the $a$-core-EP index of $c$, and an example is constructed to illustrate this inequality. A list of criteria for the $(b,c)$-core-EP-inverse of $a$ are established, among these, it is shown in Theorems \ref{3.8} and \ref{new one-sided} that $a$ is $(b,c)$-core-EP invertible with $k\in {\rm I}_{(b,c)}(a)$ iff the product $((ca)^kc)^{(1,3)}ca$ is $((ca)^{k}b,(ca)^{k}c)$-invertible and $(ca)^{k}c$ (resp., $(ca)^{k+1}$ or $(ca)^{k+1}b$) is $\{1,3\}$-invertible iff $a$ is left $((ca)^{k}b,c)$-invertible and $(ca)^{k}c$ is $\{1,3\}$-invertible with $abr\in (ca)^{k}c\{1,3\}$ for some $r\in S$. Then an explicit formula of the $(b,c)$-core-EP inverse of $a$ is given via a class of Drazin's $(b,c)$-inverse and a $\{1,3\}$-inverse. As a consequence, a corollary on $(b,c)$-core inverse is given. This improves the classical result in \cite[Theorem 2.7]{Zhu2024}. Another accomplishment of this section is to outline and clarify the relations between the $(b,c)$-inverse, the core inverse, the core-EP inverse, the $w$-core inverse, the $(b,c)$-core inverse and our defined $(b,c)$-core-EP inverse, which enriches the known literature in this area. When $S$ is a $*$-ring, the $(b,c)$-core-EP inverse is described in terms of the direct sum decomposition of the left (right) annihilators and ideals. Section 4 focuses on presenting the criterion for the $(b,c)$-core inverse by the rank condition in the context of complex matrices. We then establish the unique solution to the system of matrix equations and the constrained matrix approximation problem in the Euclidean norm. We close the paper by exhibiting the notion of the dual $(b,c)$-core-EP inverse of $a$ in Section 5.

\section{The $(b,c)$-core-EP inverse}

We start by briefly introducing the notion of the $(b,c)$-core-EP inverse of an element by using Green's relations for a $*$-monoid. In this case, the techniques such as matrix decompositions, dimensional analysis and spectral theory cannot be used. The results in this paper are proved by a semigroup theoretical and ring theoretical languages.

\begin{definition}\label{3.1}
Let $a,b,c\in S$. We say that $a$ is $(b,c)$-core-EP invertible if there exist some $x\in S$ and $k\in \mathbb{N}$ such that
\begin{center}
$cax(ca)^{k}c=(ca)^{k}c$, $x{\mathcal R}(ca)^{k}b$ and $x{\mathcal L}((ca)^{k}c)^{*}$.
\end{center} Such an $x$ is called a $(b,c)$-core-EP inverse of $a$.
\end{definition}

By the symbol ${\rm I}_{(b,c)}(a)$ we denote the set consisting of all $k\in \mathbb{N}$ satisfying Definition \ref{3.1} above, the smallest integer $k\in {\rm I}_{(b,c)}(a)$ is called the $(b,c)$-core-EP index of $a$ and is denoted by ${\rm i}_{(b,c)}(a)$.

When ${\rm I}_{(b,c)}(a)=\{0\}$, it turns out that the $(b,c)$-core-EP inverse is exactly the $(b,c)$-core inverse.

We remind the reader that core inverses, core-EP inverses, and $w$-core-EP inverses are instances of $(b,c)$-core-EP inverses. In Theorem \ref{four} below, we will give a detailed description about their connections.

Firstly, we focus here on introducing the following binary relations. To our knowledge, there was no previous work in this line.

\begin{definition} \label{new Def} Let $a,b\in S$.

\emph{(i)} We say that $a$ is below $b$ under $\mathcal R^{\rm o}$, write $a\leq_{\mathcal R^{\rm o}} b$, if $bs=bt$ implies $as=at$ for any $s,t\in S$.

\emph{(ii)} We say that $a$ is below $b$ under $\mathcal L^{\rm o}$, write $a\leq_{\mathcal L^{\rm o}}b$, if $sb=tb$ implies $sa=ta$ for any $s,t\in S$.

\emph{(iii)} We say that $a$ is below $b$ under $\mathcal H^{\rm o}$, write $a\leq_{\mathcal H^{\rm o}}b$, if $a\leq_{\mathcal L^{\rm o}}b$ and $a\leq_{\mathcal R^{\rm o}}b$.
\end{definition}

\begin{proposition} Any one of the three binary relations in Definition {\rm \ref{new Def}} is a preorder.
\end{proposition}

\begin{proof}
Take (i) for example, as (ii) and (iii) could be proved by a similar way.

To prove that $a\leq_{\mathcal R^{\rm o}}b$ is a preorder. It suffices to show (1) the reflexivity ($a\leq_{\mathcal R^{\rm o}}a$) and (2) the transitivity ($a\leq_{\mathcal R^{\rm o}}b$ and $b\leq_{\mathcal R^{\rm o}}c$ imply $a\leq_{\mathcal R^{\rm o}}c$) for any $a,b,c\in S$.

(1) The reflexivity is clear.

(2) Assume $a\leq_{\mathcal R^{\rm o}}b$ and $b\leq_{\mathcal R^{\rm o}}c$. Say $cs=ct$ for any $s,t\in S$, then $bs=bt$ since $b\leq_{\mathcal R^{\rm o}}c$, which concludes $as=at$ by $a\leq_{\mathcal R^{\rm o}}b$. So, $a\leq_{\mathcal R^{\rm o}}c$.
\hfill$\Box$
\end{proof}

It should be noted that the preorders in Definition \ref{new Def} are extensions of Green's preorders, that is (a) $a\leq_{\mathcal R}b \Rightarrow a\leq_{\mathcal L^0}b$; (b) $a\leq_{\mathcal L}b \Rightarrow a\leq_{\mathcal R^0}b$; (c) $a\leq_{\mathcal H}b \Rightarrow a\leq_{\mathcal H^0}b$. The reverse directions of (a), (b) and (c) don't hold in general without additional conditions. By enabling $b\in bSb$, in which case, $b$ is called regular in the sense of von Neumann, then the converse directions of (a), (b) and (c) hold. Any $x\in S$ satisfying $b=bxb$ is called an inner inverse or \{1\}-inverse of $b$, and is usually denoted by $b^-$. By $b\{1\}$ we denote the set of all inner inverses of $b$.

We next establish the equivalence between Green's well known preorders and the preorders in Definition \ref{new Def} with additional assumptions.

\begin{proposition} \label{3.2}
Let $a,b\in S$ with $b$ regular. Then we have

\emph{(i)} $a\leq_{\mathcal R}b$ iff $a\leq_{\mathcal L^{\rm o}}b$.

\emph{(ii)} $a\leq_{\mathcal L}b$ iff $a\leq_{\mathcal R^{\rm o}}b$.

\emph{(iii)} $a\leq_{\mathcal H}b$ iff $a\leq_{\mathcal H^{\rm o}}b$.
\end{proposition}

\begin{proof}
We shall only prove (i) here, and leave the similar proofs of (ii) and (iii) as exercises.

For (i), the forward direction is clear. For the other direction, since $1b=bb^-b$ and $a\leq_{\mathcal L^{\rm o}}b$, we have $1a=bb^-a$, so that $a\leq_{\mathcal R}b$.
\hfill$\Box$
\end{proof}

A semigroup $S$ is regular if all its elements are regular. In such a semigroup, the minus partial order is defined as follows: $a \leq b$ iff $a^-a=a^-b$ and $aa^-=ba^-$ for some $a^{-}\in a\{1\}$. In the set ${\rm E}(S)$ consisting of all idempotents of $S$, $e \leq f$ iff $e=ef=fe$. For any $e,f\in {\rm E}(S)$, it follows that $e\leq_{\mathcal R^{\rm o}}f \Leftrightarrow e=ef$, and $e\leq_{\mathcal L^{\rm o}}f \Leftrightarrow e=fe$. In particular, $e\leq_{\mathcal H^{\rm o}}f \Leftrightarrow e=ef=fe$. So, the minus partial order $e \leq f \Leftrightarrow e\leq_{\mathcal H^{\rm o}}f$ in ${\rm E}(S)$.

As shown in \cite{Drazin2021}, Drazin defined the right annihilator of $a\in S$ as ${\rm ann}_r(a)=\{(r,s)\in S \times S:ar=as\}$, and the left annihilator of $a$ as ${\rm ann}_l(a)=\{(p,q)\in S\times S:pa=qa\}$.

We next describe the connections between annihilators in the sense of Drazin and Green's generalized preorders.

\begin{theorem} For any $a,b\in S$, we have

{\rm (i)} $a\leq_{\mathcal R^{\rm o}}b$ iff ${\rm ann}_r(b) \subseteq {\rm ann}_r(a)$.

{\rm (ii)} $a\leq_{\mathcal L^{\rm o}}b$ iff ${\rm ann}_l(b) \subseteq {\rm ann}_l(a)$.
\end{theorem}

\begin{proof}
Again, we here shall only prove (i). Then $a\leq_{\mathcal R^{\rm o}}b$ iff $bs=bt\Rightarrow as=at$ for all $s,t\in S$, iff $(s,t)\in {\rm ann}_r(b)\Rightarrow (s,t)\in {\rm ann}_r(a)$ iff ${\rm ann}_r(b) \subseteq {\rm ann}_r(a)$.
\hfill$\Box$
\end{proof}

Suppose that $S$ is a ring with unity. Then for any $a,b\in S$, we have the following interesting property about the preorders: $a\leq_{\mathcal L^{\rm o}}b$ iff $a-b\leq_{\mathcal L^{\rm o}}b$. Here a proof is given. First assume $a\leq_{\mathcal L^{\rm o}}b$, say with $xb=yb$ for any $x,y\in S$. Then $xa=ya$ and $x(a-b)=y(a-b)$, as required. Conversely, suppose $xb=yb$ for any $x,y\in S$. Then $x(a-b)=y(a-b)$, so that $xa=x(b+a-b)=xb+x(a-b)=yb+y(a-b)=ya$. Therefore, $a\leq_{\mathcal L^{\rm o}}b$. We can also deduce that $a\leq_{\mathcal R^{\rm o}}b$ iff $a-b\leq_{\mathcal R^{\rm o}}b$; and $a\leq_{\mathcal H^{\rm o}}b$ iff $a-b\leq_{\mathcal H^{\rm o}}b$.

Another three binary relations are defined as follows: (1) $a{\mathcal R^{\rm o}}b$: $bs=bt$ iff $as=at$ for any $s,t\in S$, (2) $a{\mathcal L^{\rm o}}b$: $sb=tb$ iff $sa=ta$ for any $s,t\in S$, (3) $a{\mathcal H^{\rm o}}b$ iff $a{\mathcal L^{\rm o}}b$ and $a{\mathcal R^{\rm o}}b$. Not surprisingly, we have:

\begin{proposition} \label{green relations}
Let $a,b\in S$. Then we have \emph{(i)} $a{\mathcal L}b \Rightarrow a{\mathcal R^{\rm o}}b$, \emph{(ii)} $a{\mathcal R}b \Rightarrow a{\mathcal L^{\rm o}}b$, \emph{(iii)} $a{\mathcal H}b \Rightarrow a{\mathcal H^{\rm o}}b$.
\end{proposition}

The following theorem lists some criteria for the $(b,c)$-core-EP inverse in terms of properties of Green's generalized preorders.

\begin{theorem}\label{3.3}
Let $a,b,c\in S$. Then the following conditions are equivalent{\rm:}

\emph{(i)} $a$ is $(b,c)$-core-EP invertible.

\emph{(ii)} there exist some $x\in (ca)^{k}bS$ and some $k\in \mathbb{N}$ such that $cax(ca)^{k}c=(ca)^{k}c$, $x(ca)^{k+1}b=(ca)^{k}b$, $(cax)^{*}=cax$ and $xcax=x$.

\emph{(iii)} there exist some $x\in (ca)^{k}bS$ and some $k\in \mathbb{N}$ such that $cax(ca)^{k}c=(ca)^{k}c$, $x(ca)^{k+1}b=(ca)^{k}b$ and $(cax)^{*}=cax$.

\emph{(iv)} there exist some $x\in (ca)^{k}bS$ and some $k\in \mathbb{N}$ such that $cax(ca)^{k}c=(ca)^{k}c$, $(ca)^{k}b \leq_{\mathcal L^{\rm o}} x$ and $x \leq_{\mathcal R^{\rm o}}((ca)^{k}c)^{*}$.

\emph{(v)} there exists some $x\in S$ and some $k\in \mathbb{N}$ such that $cax(ca)^{k}c=(ca)^{k}c$, $x{\mathcal R}(ca)^{k}b$ and $x \leq_{\mathcal L}((ca)^{k}c)^{*}$.

\emph{(vi)} there exists some $x\in S$ and some $k\in \mathbb{N}$ such that $cax(ca)^{k}c=(ca)^{k}c$, $x{\mathcal R} (ca)^{k}b$ and $x{\mathcal R^{\rm o}}((ca)^{k}c)^{*}$.
\end{theorem}

\begin{proof}
To begin with, (ii) $\Rightarrow$ (iii) is a tautology.

(i) $\Rightarrow$ (ii) As $a$ is $(b,c)$-core-EP invertible, then there exist some $x\in S$ and $k\in \mathbb{N}$ such that $cax(ca)^{k}c=(ca)^{k}c$, $x{\mathcal R}(ca)^{k}b$ and $x{\mathcal L}((ca)^{k}c)^{*}$, and consequently $x\in (ca)^{k}bS$. Since $x{\mathcal L}((ca)^{k}c)^{*}$ and $cax(ca)^{k}c=(ca)^{k}c$, there exists some $t\in R$ such that $x=t((ca)^{k}c)^{*}=t(cax(ca)^{k}c)^*=t((ca)^{k}c)^{*}(xac)^*=x(cax)^{*}$, and hence $cax=cax(cax)^{*}=(cax)^{*}$, so that $x=x(cax)^{*}=xcax$. Again, from $x{\mathcal R}(ca)^{k}b$ and $x=xcax$, it follows that $(ca)^{k}b=xs=xcaxs=xca(ca)^{k}b=x(ca)^{k+1}b$ for some $s\in S$.

(iii) $\Rightarrow$ (iv) Given $x(ca)^{k+1}b=(ca)^{k}b$, hence $(ca)^{k}b\leq_{\mathcal R} x$, then $(ca)^{k}b \leq_{\mathcal L^{\rm o}} x$ by Proposition \ref{3.2}. It follows from $x\in (ca)^{k}bS$ and $x(ca)^{k+1}b=(ca)^{k}b$ that $x=(ca)^kbt=x(ca)^{k+1}bt=xcax=x(cax)^*=xx^*(ca)^*=x((ca)^kbt)^*(ca)^*=x((ca)^kcabt)^*=x(abt)^*((ca)^kc)^*$, i.e., $x\leq_{\mathcal L}((ca)^{k}c)^{*}$. Clearly, $x \leq_{\mathcal R^{\rm o}}((ca)^{k}c)^{*}$.

(iv) $\Rightarrow$ (v) Note that $((ca)^{k}c)^{*}=((ca)^{k}c)^{*}(cax)^{*}$. Then $x=x(cax)^{*}$ since $x \leq_{\mathcal R^{\rm o}}((ca)^{k}c)^{*}$, so that $cax=cax(cax)^{*}=(cax)^{*}$. This in turn gives $x=xcax=(xca)x$, which together with $(ca)^{k}b \leq_{\mathcal L^{\rm o}} x$ give $(ca)^{k}b=x(ca)^{k+1}b$,  and $(ca)^{k}b\leq_{\mathcal R} x \leq_{\mathcal R} (ca)^{k}b$, i.e., $x {\mathcal R}(ca)^{k}b$. Analogously, Green's preorder $x \leq_{\mathcal L}((ca)^{k}c)^{*}$ follows from $xcax=x$ and $x {\mathcal R}(ca)^{k}b$.

(v) $\Rightarrow$ (vi) Notice that $x \leq_{\mathcal L}((ca)^{k}c)^{*}$ yields $x\leq_{\mathcal R^{\rm o}}((ca)^{k}c)^{*}$. It is sufficient to prove $((ca)^{k}c)^{*} \leq_{\mathcal R^{\rm o}}x$. An argument similar to (i) $\Rightarrow$ (ii) gives $cax=(cax)^{*}$, and so that $((ca)^{k}c)^{*}=((ca)^{k}c)^{*}cax$. Say with $xs=xt$ for any $s,t\in S$, then $((ca)^{k}c)^{*}s=((ca)^{k}c)^{*}caxs=((ca)^{k}c)^{*}caxt=((ca)^{k}c)^{*}t$, hence, $((ca)^{k}c)^{*} \leq_{\mathcal R^{\rm o}}x$, as required.

(vi) $\Rightarrow$ (i) We need only to show $x \mathcal{L} ((ca)^{k}c)^{*}$. Since $x{\mathcal R^{\rm o}}((ca)^{k}c)^{*}$, by Proposition \ref{3.2}, we shall prove that $x$ and $((ca)^{k}c)^{*}$ are both regular. From $cax(ca)^{k}c=(ca)^{k}c$ and $x{\mathcal R} (ca)^{k}b$, we have $(ca)^{k}c\in ca(ca)^{k}bS(ca)^{k}c\subseteq (ca)^kcS (ca)^kc$, namely $(ca)^kc$ is regular and so is $((ca)^{k}c)^{*}$. Notice that $cax(ca)^{k}c=(ca)^{k}c$ gives $((ca)^{k}c)^{*}(cax)^{*}=((ca)^{k}c)^{*}$. Then $x=x(cax)^{*}$ since $x{\mathcal R^{\rm o}}((ca)^{k}c)^{*}$, so $cax=(cax)^{*}$ and $x=x(cax)^{*}=xcax$, so $x$ is regular, as required.
\hfill$\Box$
\end{proof}

Clearly, any $x\in S$ satisfying the conditions (ii)-(vi) in Theorem \ref{3.3} is exactly the $(b,c)$-core-EP inverse of $a$, it is $\{1,4\}$-invertible as well, and is definitely regular. We hence conclude the fact that if $a$ is $(b,c)$-core-EP invertible with $k\in {\rm I}_{(b,c)}(a)$ then $(ca)^kb$ and $(ca)^kc$ are both regular in terms of the following statement: For any $y,z\in S$ and $y\mathcal{R} z$, then $y$ is regular iff $z$ is regular.

With the help of Theorem \ref{3.3}, we deduce the following crucial uniqueness of the $(b,c)$-core-EP inverse of $a$.

\begin{theorem}\label{3.5}
Let $a,b,c\in S$. Then $a$ has at most one $(b,c)$-core-EP inverse in $S$.
\end{theorem}

\begin{proof}
Suppose $x,y\in S$ are two $(b,c)$-core-EP inverses of $a$. Then one has $xcax=x$, $x\mathcal{R}(ca)^{k}b\mathcal{R}y$, $x\mathcal{L}((ca)^{k}c)^{*}\mathcal{L}y$ and $ycay=y$ for some $k\in \mathbb{N}$. Once given $xcax=x$ and $x\mathcal{R}(ca)^{k}b\mathcal{R}y$, then there is some $s\in S$ such that $y=xs=xcaxs=xcay$. Similarly, $ycay=y$ and $x\mathcal{L}((ca)^{k}c)^{*}\mathcal{L}y$ imply $xcay=x$. Therefore, $x=y$.
\hfill$\Box$
\end{proof}

By Theorem \ref{3.5}, the $(b,c)$-core-EP inverse of $a$ is unique when it exists. Let $a_{(b,c)}^{\tiny{\textcircled{D}}}$ denote the $(b,c)$-core-EP inverse of $a$. The symbol $S_{(b,c)}^{\tiny{\textcircled{D}}}$ stands for the set of all $(b,c)$-core-EP invertible elements in $S$.

As in \cite{Anderson1974}, let the symbol ${\rm card}(A)$ stand for the cardinality of the set $A$. When $A=\{1,2,3,4\}$, then ${\rm card}(A)=4$. If $A=\mathbb{N}$ or $A=\mathbb{N^*}$, then ${\rm card}(A)=\infty$. By ${\rm I}_w(a)$ we denote the set of all $k\in \mathbb{N}$ satisfying the system of equations $x(aw)^{k+1}a=(aw)^ka$, $awx^2=x$ and $awx=(awx)^*$. One knows that if $k\in \mathbb{N}$ is in ${\rm I}_w(a)$ then so is $k+n$ for any $n\in \mathbb{N^*}$. Indeed, $x(aw)^{k+1}a=(aw)^ka$ implies $x(aw)^{k+n+1}a=(aw)^{k+n}a$. It follows that ${\rm card}({\rm I}_w(a))=\infty$.

We herein remind the reader that there are only two cases on ${\rm card}({\rm I}_{(b,c)}(a))$, where the first case is ${\rm card}({\rm I}_{(b,c)}(a))=1$, as Example \ref{3.6} shows, the second case is that if ${\rm card}({\rm I}_{(b,c)}(a)) > 1$ then ${\rm card}({\rm I}_{(b,c)}(a))=\infty$ (indeed, this set is countably infinite). We next illustrate the second case.

\begin{theorem} \label{add}
For any $a,b,c\in S$, let $a$ be $(b,c)$-core-EP invertible with ${\rm i}_{(b,c)}(a)=k$. If ${\rm card}({\rm I}_{(b,c)}(a)) >1$, then for any integer $m>k$, $m$ is in the set ${\rm I}_{(b,c)}(a)$.
\end{theorem}

\begin{proof}
Since $k={\rm i}_{(b,c)}(a)$, the $(b,c)$-core-EP index of $a$, there is some $x\in (ca)^{k}bS$ such that $cax(ca)^{k}c=(ca)^{k}c$ and $x(ca)^{k+1}b=(ca)^{k}b$. Thereby we have $(ca)^{k}c\leq_\mathcal{R} (ca)^{k+1}b\leq_\mathcal{R} (ca)^{k}c$, i.e., $(ca)^{k}c\mathcal{R} (ca)^{k+1}b$. This gives $(ca)^{m-1}c\mathcal{R}(ca)^{m}b$ for any integer $m> k$.

Suppose $t\in I_{(b,c)}(a)$ and $t>k$. Then $x\mathcal{R}(ca)^{t}b$, $x\mathcal{L}((ca)^{t}c)^{*}$, $cax(ca)^{t}c=(ca)^{t}c$ and $x(ca)^{t+1}b=(ca)^{t}b$. From $(ca)^kc=cax(ca)^kc$ and $x\in (ca)^tbS$ it follows that $(ca)^kc\leq_\mathcal{R} (ca)^{t+1}b \leq_\mathcal{R}(ca)^{k+1}c \leq_\mathcal{R} (ca)^kc$, for that would imply $(ca)^{k}c \mathcal{R}(ca)^{k+1}c$. By induction it follows that $(ca)^{k}c \mathcal{R}(ca)^{k+1}c\mathcal{R} (ca)^{k+2}c\mathcal{R}\cdots$, hence $(ca)^{k}c\mathcal{R}(ca)^{t-1}c\mathcal{R}(ca)^{t}b \mathcal{R} x\mathcal{R} (ca)^{k}b$, which together with $cax(ca)^kc=(ca)^kc$ and $x\mathcal{L}((ca)^kc)^*$ implies that $c$ is $a$-core-EP invertible with $k\in {\rm I}_a(c)$. So we conclude $k+n\in {\rm I}_a(c)$ for any $n\in \mathbb{N^*}$, hence $cax(ca)^{k+n}c=(ca)^{k+n}c$, $x\mathcal{R}(ca)^{k+n}c$ and $x\mathcal{L}((ca)^{k+n}c)^*$. Since $(ca)^{k}c\mathcal{R}(ca)^{k}b$, for that we would imply $(ca)^{k+n}c\mathcal{R}(ca)^{k+n}b$ so that $x\mathcal{R}(ca)^{k+n}b$. Therefore, $a$ is $(b,c)$-core-EP invertible with $k+n\in {\rm I}_{(b,c)}(a)$, as required.
\hfill$\Box$
\end{proof}

For any $a,b,c,w\in S$, the $(b,c)$-core-EP invertible elements may not be $w$-core-EP invertible. As previously stated, the $(b,c)$-core-EP inverse of $a$ reduces to the $a$-core-EP inverse of $c$ provided that ${\rm card}({\rm I}_{(b,c)}(a))>1$. However, if ${\rm card}({\rm I}_{(b,c)}(a))=1$, and even ${\rm I}_{(b,c)}(a)=\{0\}$, the $(b,c)$-core-EP inverse can not be simplified to the special $w$-core-EP inverse in general.

We should like to point out that in the special case where ${\rm card}({\rm I}_{(b,c)}(a)) >1$, Theorem \ref{add} leads to the following implication ``$a$ is $(b,c)$-core-EP invertible implies that $c$ is $a$-core-EP invertible.''

\begin{corollary} \label{new cor} For any $a,b,c\in S$, let $a$ be $(b,c)$-core-EP invertible and ${\rm card}({\rm I}_{(b,c)}(a)) >1$. Then $c$ is $a$-core-EP invertible. Moreover, ${\rm i}_a(c)\leq{\rm i}_{(b,c)}(a)$.
\end{corollary}

Let us now proceed with our construction to illustrate Corollary \ref{new cor}.

\begin{example}
{\rm Let $S= \mathbb{Z}_{8}$ be the semigroup of all integers modulo 8  and let involution be the identity map. Suppose $a=b=1$, $c=2 \in S$. A quick computation demonstrates that $x=0$ is the $(b,c)$-core-EP inverse of $a$ and ${\rm I}_{(b,c)}(a)=\{3,4,5,\cdots\}$. Moreover, ${\rm i}_{(b,c)}(a)=3$. Similarly, it follows that $c$ is $a$-core-EP invertible and ${\rm I}_a(c)=\{2,3,4,\cdots\}$. So, ${\rm i}_{a}(c)=2\leq 3={\rm i}_{(b,c)}(a)$. By hindsight, this example also indicates that ${\rm I}_{(b,c)}(a)$ and ${\rm I}_{a}(c)$ are equivalent since there is a bijection between the two sets.}
\end{example}

In what follows, we assume that ${\rm card}({\rm I}_{(b,c)}(a))=1$ and $0\notin {\rm I}_{(b,c)}(a)$, unless otherwise specified.

Given any $a,x\in S$, consider the following four conditions (1) $axa=a$, (2) $xax=x$, (3) $(ax)^{*}=ax$, (4) $(xa)^{*}=xa$. If $a,x\in S$ satisfy the equations $\{i_1,...,i_n\}\subseteq \{1,2,3,4\}$, then $x$ is called a $\{i_1,...,i_n\}$-inverse of $a$ and is denoted by $a^{(i_1,...,i_n)}$. By the symbol $a\{i_1,...,i_n\}$ we denote the set of all $\{i_1,...,i_n\}$-inverses of $a$. As usual, $S^{(i_1,...,i_n)}$ stands for the set of all $\{i_1,...,i_n\}$-invertible elements in $S$. The Moore--Penrose inverse of $a$ \cite{Penrose1955} is exactly its $\{1,2,3,4\}$-inverse, and is denoted by $a^{\dag}$. We denote by $S^{\dag}$ the set of all Moore--Penrose invertible elements in $S$.

As proved in \cite[Theorem 2.7]{Zhu2024}, $a$ is $(b,c)$-core invertible iff $a$ is $(b,c)$-invertible and $c$ is $\{1,3\}$-invertible, in which case, $a_{(b,c)}^{\tiny{\textcircled{\#}}}=a^{(b,c)}c^{(1,3)}$.

So, one might except, whether the $(b,c)$-core-EP invertibility of $a$ is equal to its $((ca)^kb,(ca)^kc)$-invertibility and $\{1,3\}$-invertibility of certain element, where $k\in I_{(b,c)}(a)$. This may fail without additional condition, see the following example.

\begin{example}\label{3.6}
{\rm Let $S=M_3(\mathbb{C})$ be the ring of all $3\times 3$ complex matrices and let the involution $*$ be the conjugate transpose. Assume $a=I_{3}$,
$b=
\begin{bmatrix}
0 & 0 & 0\\
0 & 0 & 0 \\
0 & 0 & 1 \\
\end{bmatrix}$,
$c=
\begin{bmatrix}
0 & 1 & 0\\
0 & 0 & 1 \\
0 & 0 & 0 \\
\end{bmatrix}
$ and $x=
\begin{bmatrix}
0 & 0 & 0\\
1 & 0 & 0 \\
0 & 0 & 0 \\
\end{bmatrix}$. A direct check shows that $\{2,3,4,\cdots\} \subsetneqq {\rm I}_{(b,c)}(a)$, i.e., $ 2 \leq m \notin {\rm I}_{(b,c)}(a)$. Also, notice that $x\in cabS$, $caxcac=cac$, $x(ca)^{2}b=cab$ and $(cax)^{*}=cax$. Thus $a$ is $(b,c)$-core-EP invertible with $I_{(b,c)}(a)=\{1\}$. However, $a$ is not $(cab,cac)$-invertible since $cac\notin (ca)^{3}bS$.}
\end{example}

Example \ref{3.6} indicates that if $a$ is $(b,c)$-core-EP invertible with $k\in {\rm I}_{(b,c)}(a)$ then it may not be $((ca)^{k}b, (ca)^{k}c)$-invertible in general (We assert that $a$ is not $(b,c)$-core invertible as well, in fact, $x\notin bS$). It is of interest to ask whether the $(b,c)$-core-EP invertibility of $a$ with $k\in {\rm I}_{(b,c)}(a)$ implies the $((ca)^{k}b, (ca)^{k}c)$-invertibility and $\{1,3\}$-invertibility of certain elements. The answer is positive. Surprisingly, the converse statement also holds as Theorem \ref{3.8} below shows.

Recall two auxiliary lemmas about important characterizations of $(b,c)$-inverses and $\{1,3\}$-inverses, which will be used in our proof of Theorem \ref{3.8}.

\begin{lemma} \label{bc inverse} {\rm \cite[Theorem 2.2]{Drazin2012}} Let $a,b,c\in S$. Then $a$ is $(b,c)$-invertible iff $b\in Scab$ and $c\in cabS$. In particular, if $b=vcab$ and $c=cabw$ for some $v,w\in S$, then $a^{(b,c)}=bw=vc$.
\end{lemma}

\begin{lemma}\label{3.7} {\rm \cite[Lemma 2.2]{Zhu2015}}
Let $a\in S$. Then

\emph{(i)} $a\in S^{(1,3)}$ iff $a\in Sa^{*}a$. In particular, if $a=xa^{*}a$ for some $x\in S$, then $x^{*}$ is a $\{1,3\}$-inverse of $a$.

\emph{(ii)} $a\in S^{(1,4)}$ iff $a\in aa^{*}S$. In particular, if $a=aa^{*}y$ for some $y\in S$, then $y^{*}$ is a $\{1,4\}$-inverse of $a$.
\end{lemma}

Given any $a,b,c\in S$, we now show that if $a$ is $(b,c)$-core-EP invertible then $a_{(b,c)}^{\tiny{\textcircled{D}}}$ is $\{1,4\}$-invertible. Suppose we are given $cax(ca)^{k}c=(ca)^{k}c$, $x{\mathcal R}(ca)^{k}b$ and $x{\mathcal L}((ca)^{k}c)^{*}$, where $x=a_{(b,c)}^{\tiny{\textcircled{D}}}$ and $k\in {\rm I}_{(b,c)}(a)$. Then there is some $t\in S$ such that $x^*=(ca)^kct=(cax(ca)^kc)t=cax((ca)^kct)=caxx^*\in Sxx^*$, so that $x\in xx^*S$, i.e., $a_{(b,c)}^{\tiny{\textcircled{D}}}$ is $\{1,4\}$-invertible. Moreover, $ca$ is a $\{1,4\}$-inverse of $a_{(b,c)}^{\tiny{\textcircled{D}}}$. This implication can also be obtained from Theorem \ref{3.3} (ii).

As shown above, if $a$ is $(b,c)$-core-EP invertible with $k\in {\rm I}_{(b,c)}(a)$, then $(ca)^{k}c$ is $\{1,3\}$-invertible. It is convenient to write $d:=((ca)^kc)^{(1,3)}ca$ in Theorem \ref{3.8}, where $((ca)^kc)^{(1,3)}\in (ca)^kc\{1,3\}$.

In the following result, we will develop the criteria for the $(b,c)$-core-EP inverse of $a$ with $k\in {\rm I}_{(b,c)}(a)$ in terms of the ${((ca)^{k}b,(ca)^{k}c)}$-inverse of $d$ and a $\{1,3\}$-inverse of $(ca)^{k}c$ (resp., $(ca)^{k+1}$ or $(ca)^{k+1}b$).

\begin{theorem}\label{3.8}
Let $a,b,c\in S$. Then the following conditions are equivalent{\rm:}

\emph{(i)} $a$ is $(b,c)$-core-EP invertible with $k\in I_{(b,c)}(a)$.

\emph{(ii)} $(ca)^{k}c\in S^{(1,3)}$ and $d$ is ${((ca)^{k}b,(ca)^{k}c)}$-invertible.

\emph{(iii)}  $(ca)^{k+1}\in S^{(1,3)}$ and $d$ is ${((ca)^{k}b,(ca)^{k}c)}$-invertible.

\emph{(iv)} $(ca)^{k+1}b\in S^{(1,3)}$ and $d$ is ${((ca)^{k}b,(ca)^{k}c)}$-invertible.

In this case, we have
\begin{eqnarray*}
a_{(b,c)}^{\tiny{\textcircled{D}}}&=&d^{((ca)^{k}b,(ca)^{k}c)}((ca)^{k}c)^{(1,3)}\\
&=&d^{((ca)^{k}b,(ca)^{k}c)}a((ca)^{k+1})^{(1,3)}\\
&=&(ca)^{k}b((ca)^{k+1}b)^{(1,3)},
\end{eqnarray*}
where $d=((ca)^kc)^{(1,3)}ca$.
\end{theorem}

\begin{proof}
(i) $\Rightarrow$ (ii) Since $a$ is $(b,c)$-core-EP invertible with $k\in I_{(b,c)}(a)$, it follows that $(ca)^kc=cax(ca)^kc\in Sx(ca)^kc=S((ca)^kc)^*(ca)^kc$ for some $x\in S$, so that $\{1,3\}$-invertible by Lemma \ref{3.7}. It next suffices to prove the other implication. Given (i), then for any $((ca)^kc)^-\in (ca)^kc\{1\}$, we have at once $(ca)^{k}c= cax(ca)^{k}c\in (ca)^{k+1}bS=(ca)^kcd(ca)^kbS$. A similar argument gives $(ca)^{k}b\in S(ca)^kcd(ca)^kb$ in terms of $x(ca)^{k+1}b=(ca)^{k}b$, so that $d$ is $((ca)^{k}b,(ca)^{k}c)$-invertible by Lemma \ref{bc inverse}.

(ii) $\Rightarrow$ (iii) From $(ca)^{k}c\in S^{(1,3)}$, it follows that $(ca)^{k}c\in S((ca)^{k}c)^{*}(ca)^{k}c$, and $(ca)^{k+1}\in S((ca)^{k}c)^{*}(ca)^{k+1}\subseteq  S((ca)^{k+1}b)^{*}(ca)^{k+1}\subseteq S((ca)^{k+1})^{*}(ca)^{k+1}$ since $d\in S^{((ca)^{k}b,(ca)^{k}c)}$ gives $(ca)^{k}c\in (ca)^{k+1}bS$, as required.

(iii) $\Rightarrow$ (iv) The condition $d\in S^{((ca)^{k}b,(ca)^{k}c)}$ implies $(ca)^{k}c\in (ca)^{k+1}bS$, and consequently $(ca)^{k+1}\in (ca)^{k+1}bSa\subseteq (ca)^{k+1}bS$. Notice that $(ca)^{k+1}\in S((ca)^{k+1})^{*}(ca)^{k+1}$. Then $(ca)^{k+1}b\in S((ca)^{k+1}b)^{*}(ca)^{k+1}b$, i.e., $(ca)^{k+1}b\in S^{(1,3)}$.

(iv) $\Rightarrow$ (i) Since $d$ is ${((ca)^{k}b,(ca)^{k}c)}$-invertible, we have that $(ca)^{k}b\in S(ca)^{k+1}b$ and $(ca)^{k}c\in (ca)^{k+1}bS$, so that there exist some $s,t\in S$ such that $(ca)^{k}b= s(ca)^{k+1}b$ and $(ca)^{k}c= (ca)^{k+1}bt$. Write $x:=(ca)^{k}b((ca)^{k+1}b)^{(1,3)}\in (ca)^{k}bS$, we next show that $x$ is the $(b,c)$-core-EP inverse of $a$ by the following three steps.

(1) $cax=(ca)^{k+1}b((ca)^{k+1}b)^{(1,3)}=(cax)^{*}$.

(2) We have
\begin{eqnarray*}
cax(ca)^kc&=&(ca)^{k+1}b((ca)^{k+1}b)^{(1,3)}(ca)^kc\\
&=&(ca)^{k+1}b((ca)^{k+1}b)^{(1,3)}(ca)^{k+1}bt\\
&=&(ca)^{k+1}bt\\
&=&(ca)^kc.
\end{eqnarray*}

(3) It follows that
\begin{eqnarray*}
x(ca)^{k+1}b&=&(ca)^{k}b((ca)^{k+1}b)^{(1,3)}(ca)^{k+1}b\\
 &=&s(ca)^{k+1}b((ca)^{k+1}b)^{(1,3)}(ca)^{k+1}b\\
 &=&(ca)^{k}b.
\end{eqnarray*}

Next, it is straightforward to check that $y=d^{((ca)^{k}b,(ca)^{k}c)}((ca)^{k}c)^{(1,3)}\in (ca)^{k}bS$ is the $(b,c)$-core-EP inverse of $a$. Write $z=d^{((ca)^{k}b,(ca)^{k}c)}$, then $(ca)^kcdz=(ca)^kc$ and $z=(ca)^{k}bt$ for some $t\in S$. We here get $caz=(ca)^kc$ as
\begin{eqnarray*}
 caz&=&(ca)^{k+1}bt=(ca)^kc((ca)^kc)^{(1,3)}(ca)^kcabt\\
 &=&(ca)^kc((ca)^kc)^{(1,3)}ca(ca)^kbt=(ca)^kcdz\\
 &=&(ca)^kc.
 \end{eqnarray*}
It follows that $cay=caz((ca)^kc)^{(1,3)}=(cay)^*$, $cay(ca)^kc=(ca)^kc$ and $y(ca)^{k+1}b=(ca)^kb$, as required.

Note that $a((ca)^{k+1})^{(1,3)}\in (ca)^{k}c\{1,3\}$. Then $d^{((ca)^{k}b,(ca)^{k}c)}a((ca)^{k+1})^{(1,3)}=d^{((ca)^{k}b,(ca)^{k}c)}((ca)^{k}c)^{(1,3)}$.  \hfill$\Box$
\end{proof}

As previously mentioned in Lemma \ref{bc inverse}, $a$ is $(b,c)$-invertible iff $c\in c(c^-ca)bS$ and $b\in Sc(c^{-}ca)b$, iff $c^-ca$ is $(b,c)$-invertible, provided that $c^-\in c\{1\}$.

An immediate consequence of Theorem \ref{3.8} is the following corollary.

\begin{corollary} \label{Mosic's result} {\rm \cite[Theorem 2.7]{Zhu2024}}
Let $a,b,c\in S$. The following statements are equivalent{\rm :}

\emph{(i)} $a\in S_{(b,c)}^{\tiny\textcircled{\tiny{\#}}}$.

\emph{(ii)} $a\in S^{(b,c)}$ and $c\in S^{(1,3)}$.

\emph{(iii)} $a\in S^{(b,c)}$ and $ca\in S^{(1,3)}$.

\emph{(iv)} $a\in S^{(b,c)}$ and $cab\in S^{(1,3)}$.

In this case, $a_{(b,c)}^{\tiny\textcircled{\tiny{\#}}}=a^{(b,c)}c^{(1,3)}=a^{(b,c)}a(ca)^{(1,3)}=b(cab)^{(1,3)}$.
\end{corollary}

The next result describes how the $(b,c)$-core-EP inverse is related to Drazin's $(b,c)$-inverse.

\begin{theorem}\label{3.16}
Let $a,b,c\in S$. Then $a$ is $(b,c)$-core-EP invertible with $k\in {\rm I}_{(b,c)}(a)$ iff there exist some $x\in S$ and some $k\in \mathbb{N}$ such that $xcax=x$, $xS=(ca)^{k}bS$ and $Sx=S((ca)^{k}c)^{*}$. In this case, the $(b,c)$-core-EP inverse of $a$ coincides with the $((ca)^kb, ((ca)^kc)^*)$-inverse of $ca$.
\end{theorem}

\begin{proof}
The ``only if'' part follows by Theorem \ref{3.3} (i) $\Rightarrow$ (ii).

The ``if'' part: To prove that $a$ is $(b,c)$-core-EP invertible, it suffices to show $(ca)^{k}c=cax(ca)^{k}c$. As $Sx=S((ca)^{k}c)^{*}$, then there exists some $t\in S$ such that $((ca)^{k}c)^{*}=tx=t(xcax)=((ca)^{k}c)^{*}cax$, and thus $(ca)^{k}c=(cax)^{*}(ca)^{k}c$. Post-multiplying $(ca)^{k}c=(cax)^{*}(ca)^{k}c$ by $ab$ gives $(ca)^{k+1}b=(cax)^{*}(ca)^{k+1}b$. Similarly, $xS=(ca)^{k}bS$ implies $x=(ca)^kbs$ for some $s\in S$, we have at once $cax=(ca)^{k+1}bs$, so that $cax=(ca)^{k+1}bs=(cax)^{*}(ca)^{k+1}bs=(cax)^{*}cax=(cax)^*$, which in turn gives $(ca)^{k}c=cax(ca)^{k}c$.
\hfill$\Box$
\end{proof}

Combining Lemma \ref{bc inverse} and Theorem \ref{3.16}, we get the following proposition.

\begin{proposition} \label{intersetion} Let $a,b,c\in S$. Then $a$ is $(b,c)$-core-EP invertible with $k\in {\rm I}_{(b,c)}(a)$ iff $(ca)^kb\in S((ca)^kc)^*(ca)^{k+1}b$ and $((ca)^kc)^*\in ((ca)^kc)^*(ca)^{k+1}bS$ for some $k\in \mathbb{N}$.
\end{proposition}

Following the initial work by Drazin \cite{Drazin2016} in a semigroup $S$, for any $a,b,c\in S$, the element $a$ is called left $(b,c)$-invertible if $b\in Scab$, i.e., if there exists some $x\in Sc$ such that $b=xab$. By $a_l^{(b,c)}$ we denote a left $(b,c)$-inverse of $a$. The element $a$ is right $(b,c)$-invertible if $c\in cabS$, i.e., if there exists some $x\in bS$ such that $c=cax$. By $a_r^{(b,c)}$ we denote a right $(b,c)$-inverse of $a$. It is well known that $a$ is $(b,c)$-invertible iff $a$ is both left and right $(b,c)$-invertible.

We assume that $a$ is $(b,c)$-core-EP invertible with $k\in {\rm I}_{(b,c)}(a)$. Then, by Proposition \ref{intersetion}, we have $(ca)^kb\in S(ca)^{k+1}b=Sca(ca)^kb$, i.e., $a$ is left $((ca)^kb,c)$-invertible. Also, from $((ca)^kc)^*\in ((ca)^kc)^*(ca)^{k}cabS$, we get that $(ca)^kc\in S(ab)^*((ca)^kc)^*(ca)^{k}c \subseteq S((ca)^kc)^*(ca)^{k}c$, which implies that $(ca)^kc$ is \{1,3\}-invertible and $abs^*\in (ca)^kc\{1,3\}$, where $s\in S$ satisfies $(ca)^kc=s(ab)^*((ca)^kc)^*(ca)^{k}c$. To summarize, if $a$ is $(b,c)$-core-EP invertible with $k\in {\rm I}_{(b,c)}(a)$, then $a$ is left $((ca)^kb,c)$-invertible and $abs^*\in (ca)^kc\{1,3\}$, where $s\in S$ satisfies $(ca)^kc=s(ab)^*((ca)^kc)^*(ca)^{k}c$.

It is natural to inquire whether the converse statement above holds, i.e., whether $a$ is left $((ca)^kb,c)$-invertible and $abr\in (ca)^kc\{1,3\}$ for some $r\in S$ can imply the $(b,c)$-core-EP invertibility of $a$. The following result answers this question in the positive.

\begin{theorem} \label{new one-sided} Let $a,b,c\in S$. Then $a$ is $(b,c)$-core-EP invertible with $k\in {\rm I}_{(b,c)}(a)$ iff there exists some $k\in \mathbb{N}$ such that $a$ is left $((ca)^kb,c)$-invertible and $(ca)^kc$ is $\{1,3\}$-invertible with $abr\in (ca)^kc\{1,3\}$ for some $r\in S$.
\end{theorem}

\begin{proof}
It suffices to prove the ``if'' part. As $abr\in (ca)^kc\{1,3\}$ for some $r\in S$, we have that $(ca)^kc=((ca)^{k+1}br)^*(ca)^kc=r^*((ca)^{k+1}b)^*(ca)^kc\in S((ca)^{k+1}b)^*(ca)^kc$. Thus $((ca)^kc)^*\in ((ca)^kc)^*(ca)^{k+1}bS$.

Since $a$ is left $((ca)^kb,c)$-invertible, $(ca)^kb\in S(ca)^{k+1}b$. By the equivalence $(ca)^kc \in S^{(1,3)}\Leftrightarrow S(ca)^kc=S((ca)^kc)^*(ca)^kc$, we have at once $(ca)^kb\in S(ca)^{k+1}b=S(ca)^kc(ab)=S((ca)^kc)^*(ca)^kc(ab)=S((ca)^kc)^*(ca)^{k+1}b$. Therefore, by Proposition \ref{intersetion}, $a$ is $(b,c)$-core-EP invertible with $k\in {\rm I}_{(b,c)}(a)$.
\hfill$\Box$
\end{proof}

For any $a,b,c\in S$, if $a$ is $(b,c)$-core-EP invertible with $k\in {\rm I}_{(b,c)}(a)$ then $a$ is right $(b,(ca)^kc)$-invertible. Indeed, it follows from Theorem \ref{new one-sided} that $(ca)^kc$ is \{1,3\}-invertible and $abr\in (ca)^kc\{1,3\}$ for some $r\in S$. This gives $(ca)^kc=(ca)^kc abr(ca)^kc\in (ca)^kcabS$, i.e., $a$ is right $(b,(ca)^kc)$-invertible.

We herein note that if $a$ is $(b,c)$-core-EP invertible with $k\in {\rm I}_{(b,c)}(a)$ then $a$ is left $(a_{(b,c)}^{\tiny{\textcircled{D}}},c)$-invertible. Indeed, as shown previously, we get $(ca)^kb\in Sca(ca)^kb$ and hence we have $a_{(b,c)}^{\tiny{\textcircled{D}}}\in Scaa_{(b,c)}^{\tiny{\textcircled{D}}}$ since $a_{(b,c)}^{\tiny{\textcircled{D}}}\in (ca)^kbS$, as required.

In light of the previous theorem, we state a criterion for the $(b,c)$-core inverse of $a$, which improves the result in \cite[Theorem 2.7]{Zhu2024}. Quite surprisingly, when the $(b,c)$-inverse of $a$ is weaken to its left $(b,c)$-inverse, it is still possible to derive the formula of the $(b,c)$-core inverse.

\begin{corollary} \label{cor new} Let $a,b,c\in S$. Then $a$ is $(b,c)$-core invertible iff $a$ is left $(b,c)$-invertible and $c$ is $\{1,3\}$-invertible with $abs\in c\{1,3\}$ for some $s\in S$. In this case, $a_{(b,c)}^{\tiny{\textcircled{\#}}}=a_l^{(b,c)}c^{(1,3)}=br^*$, where $r\in S$ satisfies $c=r(cab)^*c$.
\end{corollary}

\begin{proof} We first show that $x=a_l^{(b,c)}c^{(1,3)}=br^*\in bS$ is the $(b,c)$-core inverse of $a$, where $r\in S$ satisfies $c=r(cab)^*c$. Notice that $c=r(cab)^*c=r(ab)^*c^*c\in Sc^*c$ implies $abr^*\in c\{1,3\}$. Then $a_l^{(b,c)}c^{(1,3)}=a_l^{(b,c)}abr^*=(a_l^{(b,c)}ab)r^*=br^*$. Then we have

(1) Since $c=r(cab)^*c$ leads to $abr^*\in c\{1,3\}$, one has $cax=cabr^*=(cabr^*)^*=(cax)^*$.

(2) $caxc=(cax)^*c=(cabr^*)^*c=r(cab)^*c=c$.

(3) $xcab= a_l^{(b,c)}c^{(1,3)}cab=scc^{(1,3)}cab=scab=a_l^{(b,c)}ab=b$ since $a_l^{(b,c)}\in Sc$ gives $a_l^{(b,c)}=sc$ for some $s\in S$.
\hfill$\Box$
\end{proof}

From \cite{Zhu2023} where for any $a,w\in S$, $a$ is called $w$-core invertible if there exists some $x\in S$ such that $xawa=a$, $awx^2=x$ and $awx=(awx)^*$. Such an $x$, when exists, denoted by $a_{w}^{\tiny{\textcircled{\#}}}$, is called the $w$-core inverse of $a$. Specially, $a$ is called core invertible if it is $1$-core invertible or $a$-core invertible.

It follows from \cite[Theorem 2.2]{Mosic2023} that $a\in R$ is $w$-core-EP invertible iff there exists some $x\in R$ such that $awx(aw)^{k}a=(aw)^{k}a$, $xR=(aw)^{k}aR$ and $Rx=R((aw)^{k}a)^{*}$ for some $k\in \mathbb{N}$. As a consequence, for any $a,b\in R$, the $a$-core-EP inverse of $b$ is exactly the $(b,b)$-core-EP inverse of $a$. In a more general $S$, we say that $b$ is $a$-core-EP invertible if $a$ is $(b,b)$-core-EP invertible, i.e., there exist some $x\in S$ and $k\in \mathbb{N}$ such that $bax(ba)^kb=(ba)^kb$, $xS=(ba)^{k}bS$ and $Sx=S((ba)^{k}b)^{*}$. Hence, the notion of the $w$-core-EP inverse can be extended from a $*$-ring $R$ to a more general $*$-monoid $S$.

We digress for a paragraph to discuss the unit-regularity of the $w$-core inverse of $a$. Recall that an element $a\in R$ is unit-regular (see, e.g., \cite{Lam2014}) if it has an inner inverse which is a unit of $R$. This is equivalent to the existence of an idempotent $e$ and a unit $u$ such that $a=eu$. We conclude that $a_{w}^{\tiny{\textcircled{\#}}}$ is unit-regular. Indeed, if $a_{w}^{\tiny{\textcircled{\#}}}$ exists, then $u:=aw+1-aa^{(1,3)}$ is a unit by \cite{Zhu2023}. Then $ua=awa$, so that $ua_{w}^{\tiny{\textcircled{\#}}}=awa_{w}^{\tiny{\textcircled{\#}}}$. This gives at once $a_{w}^{\tiny{\textcircled{\#}}}=u^{-1}awa_{w}^{\tiny{\textcircled{\#}}}$. Notice that $(awa_{w}^{\tiny{\textcircled{\#}}})^2=awa_{w}^{\tiny{\textcircled{\#}}}$. Then $a_{w}^{\tiny{\textcircled{\#}}}$ is unit-regular.

Let us now turn to our main topic. In the following result, we show several other kinds of generalized inverses such as the core inverse, the core-EP inverse, the $w$-core inverse and the $w$-core-EP inverse are special cases of $(b,c)$-core-EP inverses.

\begin{theorem} \label{four} Let $a,b,c\in S$. Then we have

\emph{(i)} $a$ is core invertible iff $1$ is $(a,a)$-core-EP invertible with ${\rm I}_{(a,a)}(1)=\{0\}$, iff $a$ is $(a,a)$-core-EP invertible with ${\rm I}_{(a,a)}(a)=\{0\}$. In this case, $a^{\tiny{\textcircled{\#}}}=1_{(a,a)}^{\tiny{\textcircled{\#}}}=aa_{(a,a)}^{\tiny{\textcircled{\#}}}$.

\emph{(ii)} $a$ is core-EP invertible iff $a$ is $(1,1)$-core-EP invertible. In this case, $a^{\tiny{\textcircled{D}}}=a_{(1,1)}^{\tiny{\textcircled{D}}}$.

\emph{(iii)} For any $w\in S$, $a$ is $w$-core invertible iff $w$ is $(a,a)$-core-EP invertible with ${\rm I}_{(a,a)}(w)=\{0\}$. In this case, $a_{w}^{\tiny{\textcircled{\#}}}=w_{(a,a)}^{\tiny{\textcircled{\#}}}$.

\emph{(iv)} For any $w\in S$, $a$ is $w$-core-EP invertible with $k\in {\rm I}_w(a)$ iff $w$ is $(a,a)$-core-EP invertible with $k\in {\rm I}_{(a,a)}(w)$, iff $w$ is $((aw)^ka,(aw)^ka)$-core invertible for some $k\in \mathbb{N}$. In this case, $a_{w}^{\tiny{\textcircled{D}}}=w_{(a,a)}^{\tiny{\textcircled{D}}}=w_{((aw)^ka,(aw)^ka)}^{\tiny{\textcircled{\#}}}$.
\end{theorem}

\begin{proof}

(i) It is known that the $(b,c)$-core-EP inverse of $a$ with ${\rm I}_{(b,c)}(a)=\{0\}$ is exactly its $(b,c)$-core inverse. These two equivalences follow immediately from \cite[Theorem 2.18 (ii)]{Zhu2024}. So, $a^{\tiny{\textcircled{\#}}}=1_{(a,a)}^{\tiny{\textcircled{\#}}}=aa_{(a,a)}^{\tiny{\textcircled{\#}}}$.

(ii) Suppose $a$ is core-EP invertible. Then, by \cite[Definition 1]{Gao2013}, there exist some $x\in S$ and some $k\in \mathbb{N^*}$ such that $xa^{k+1}=a^k$, $ax^2=x$ and $ax=(ax)^*$. Then $x=ax^2=a(ax^2)x=a^2x^3=\cdots=a^kx^{k+1}\in a^kS$, $a^k=xa^{k+1}=ax^2a^{k+1}=ax(xa^{k+1})=axa^k$, which combine $xa^{k+1}=a^k$ and $ax=(ax)^*$ to imply that $a$ is $(1,1)$-core-EP invertible for some $k\in {\rm I}_{(1,1)}(a)$. Conversely, if $a$ is $(1,1)$-core-EP invertible with $k\in {\rm I}_{(1,1)}(a)$, then there is some $x\in a^kS$ such that $axa^k=a^k$, $xa^{k+1}=a^k$ and $ax=(ax)^*$. We have at once $x=a^kt$ for some $t\in S$, so that $x=a^kt=(axa^k)t=ax(a^kt)=ax^2$. Hence, $a$ is core-EP invertible.

(iii) It follows from \cite[Theorem 2.18 (i)]{Zhu2024}, since $w$ is $(a,a)$-core-EP invertible with ${\rm I}_{(a,a)}(w)=\{0\}$ iff $w$ is $(a,a)$-core invertible. Moreover, the two generalized inverses coincide with each other.

(iv) It follows that $a$ is $w$-core-EP invertible with $k\in {\rm I}_w(a)$ iff there are some $x\in S$ and $k\in \mathbb{N}$ such that $awx^{2}=x$, $x(aw)^{k+1}a=(aw)^ka$ and $(awx)^{*}=awx$, iff there exist some $x\in S$ and $k\in \mathbb{N}$ such that $x\in (aw)^kaS$, $awx(aw)^ka=(aw)^ka$, $x(aw)^{k+1}a=(aw)^ka$ and $(awx)^{*}=awx$ iff $w$ is $(a,a)$-core-EP invertible with $k\in {\rm I}_{(a,a)}(w)$. We conclude from Corollary \ref{Mosic's result} that $a$ is $w$-core-EP invertible with $k\in {\rm I}_w(a)$ iff $w$ is $((aw)^ka,(aw)^ka)$-core invertible for some $k\in \mathbb{N}$. Moreover, $a_{w}^{\tiny{\textcircled{D}}}=w_{(a,a)}^{\tiny{\textcircled{D}}}=w_{((aw)^ka,(aw)^ka)}^{\tiny{\textcircled{\#}}}$.
\hfill$\Box$
\end{proof}

At the end of the section, we deal with the direct-sum decomposition of rings, which is useful for investigating the criterion of several classes of generalized inverses. For any $a\in R$, the right annihilator of $a$ is defined by $a^0=\{x\in R:ax=0\}$ and the left annihilator of $a$ is defined by ${^0}a=\{x \in R:xa =0\}$. As stated in \cite{Drazin2012}, the well known criteria of $(b,c)$-inverses are described in terms of properties of the left (right) annihilators and ideals. Precisely, it was shown that $a\in R^{(b,c)}$ iff $R=abR\oplus c^0=Rca\oplus {^0}b$, iff $R=abR+ c^0=Rca+ {^0}b$. This naturally leads to consider the behaviour of the $(b,c)$-core-EP inverse in rings.

Combining Theorem \ref{3.16} and Drazin's equivalences for the $(b,c)$-inverse leads immediately to following characterization of $(b,c)$-core-EP inverses in $R$.

\begin{theorem} \label{direct sum}
Let $a,b,c\in R$. The following statements are equivalent{\rm:}

\emph{(i)} $a$ is $(b,c)$-core-EP invertible with $k\in {\rm I}_{(b,c)}(a)$.

\emph{(ii)}  $R=(ca)^{k+1}bR\oplus (((ca)^{k}c)^*)^0=R((ca)^{k}c)^*ca\oplus{^0((ca)^{k}b)}$ for some $k\in \mathbb{N}$.

\emph{(iii)} $R=(ca)^{k+1}bR+ (((ca)^{k}c)^*)^0=R((ca)^{k}c)^*ca+{^0((ca)^{k}b)}$ for some $k\in \mathbb{N}$.
\end{theorem}

Let ${\rm I}_{(b,c)}(a)=\{0\}$ in Theorem \ref{direct sum}. Then $a$ is $(b,c)$-core invertible iff $R=cabR\oplus (c^*)^0=Rc^*ca\oplus {^0}b$ iff $R=cabR+ (c^*)^0=Rc^*ca+{^0}b$. Let us show that $R=cabR\oplus (c^*)^0=Rc^*ca\oplus {^0}b$ can be simplified to $R=cabR\oplus (c^*)^0=Rca\oplus {^0}b$. First, suppose $R=cabR\oplus (c^*)^0=Rc^*ca\oplus {^0}b$, hence $a$ is $(b,c)$-core invertible. Then $c$ is \{1,3\}-invertible, and by Lemma \ref{3.7}, $Rc=Rc^*c$, so that $R=cabR\oplus (c^*)^0=Rca\oplus {^0}b$ and $R=cabR+ (c^*)^0=Rca+{^0}b$. Conversely, if $R=cabR\oplus (c^*)^0=Rca\oplus{^0}b$, then $c^*\in c^*cabR\subseteq c^*cR$, i.e., $Rc \subseteq Rc^*c\subseteq Rc$. Hence $R=cabR\oplus (c^*)^0=Rc^*ca\oplus{^0}b$.

\begin{corollary} {\rm \cite[Theorem 3.2]{Zhu2024}}
Let $a,b,c\in R$. The following statements are equivalent{\rm:}

\emph{(i)} $a$ is $(b,c)$-core invertible.

\emph{(ii)} $R=cabR\oplus (c^*)^0=Rca\oplus {^0}b$.

\emph{(iii)} $R=cabR+ (c^*)^0=Rca+{^0}b$.
\end{corollary}

\section{Applications to the ring of complex matrices}

The main goal of this section is to investigate the $(b,c)$-core-EP inverse of $a$ from $S$ to $\mathbb{C}_{n,n}$, the ring of all $n \times n$ complex matrices.

Hereafter, abusing notation slightly, every vector of the space $\mathbb{C}^{n}$ is considered as a column vector. For any $A\in \mathbb{C}_{n,n}$, we denote by ${\rm rk}(A)$ the rank of $A$. The range and the null space of $A$ are respectively defined as $\mathcal{R}(A)=\{Ax:x\in\mathbb{C}^{n}\}$ and $\mathcal{N}(A)=\{x\in\mathbb{C}^{n}:Ax=0\}$. By the symbol $P_{A}$ we denote the orthogonal projector $P$ onto $\mathcal{R}(A)$. The orthogonal projector is closely connected with generalized inverses, for instance, $P_A=AA^\dag$ and $P_{A^*}=A^\dag A$ are orthogonal projectors onto $\mathcal{R}(A)$ and $\mathcal{R}(A^*)$, respectively.

With what we have done in Sec. 3, the property of the $(B,C)$-core-EP inverse of $A$ is already in hand. It is concluded by Theorem \ref{intersetion} that $A$ is $(B,C)$-core-EP invertible with $k\in {\rm I}_{(B,C)}(A)$ iff $ \mathcal{N}(((CA)^kC)^*(CA)^{k+1}B) = \mathcal{N}((CA)^kB)$ and $\mathcal{R}(((CA)^kC)^*)= \mathcal{R}(((CA)^kC)^*(CA)^{k+1}B)$ for some $k\in \mathbb{N}$ iff ${\rm rk}((CA)^kC)={\rm rk}((CA)^kB)={\rm rk}(((CA)^kC)^*(CA)^{k+1}B)$ for some $k\in \mathbb{N}$.

Given any $X\in \mathbb{C}_{n,n}$, by ${\rm rk}(X^*XAB)\leq {\rm rk}(XAB)={\rm rk}((X^\dag)^*X^*XAB)\leq {\rm rk}(X^*XAB)$, it follows that ${\rm rk}(((CA)^kC)^*(CA)^{k+1}B)= {\rm rk}((CA)^{k+1}B)$, hence the criterion for the $(B,C)$-core-EP inverse of $A$ with $k\in {\rm I}_{(B,C)}(A)$ is given as follows.

\begin{theorem} \label{rank abc} Let $A,B,C\in \mathbb{C}_{n,n}$. Then $A$ is $(B,C)$-core-EP invertible with $k\in {\rm I}_{(B,C)}(A)$ iff ${\rm rk}((CA)^kB)={\rm rk}((CA)^kC)={\rm rk}((CA)^{k+1}B)$ for some $k\in \mathbb{N}$.
\end{theorem}

We shall next give a criterion for the $W$-core-EP inverse of $A$ that is a direct application of Theorems \ref{four} and \ref{rank abc}.

\begin{corollary} Let $A,W\in \mathbb{C}_{n,n}$. Then $A$ is $W$-core-EP invertible with $k\in {\rm I}_{W}(A)$ iff ${\rm rk}((AW)^kA)={\rm rk}((AW)^{k+1}A)$ for some $k\in \mathbb{N}$.
\end{corollary}

By $P_{\mathcal{M}, \mathcal{N}}$ we denote the idempotent with range $\mathcal{M}$ and null space $\mathcal{N}$. Recall that if $P_{\mathcal{M}, \mathcal{N}}$ is an idempotent of order $n$, then $\mathcal{M} \oplus\mathcal{N} = \mathbb{C}^n$.

Let $A$ be $(B,C)$-core-EP invertible and $P=A_{(B,C)}^{\tiny{\textcircled{D}}}CA$. Then $\mathbb{C}^n=\mathcal{R}(P)\oplus \mathcal{N}(P)$. Again, from $\mathcal{R}(P)\subseteq \mathcal{R}(A_{(B,C)}^{\tiny{\textcircled{D}}})=\mathcal{R}(PA_{(B,C)}^{\tiny{\textcircled{D}}})\subseteq \mathcal{R}(P)$, it follows that $\mathbb{C}^n=\mathcal{R}(A_{(B,C)}^{\tiny{\textcircled{D}}})\oplus \mathcal{N}(P)$. Since $\mathcal{R}(A_{(B,C)}^{\tiny{\textcircled{D}}})=\mathcal{R}((CA)^kB)$, we obtain $\mathbb{C}^n=\mathcal{R}((CA)^kB)\oplus \mathcal{N}(P)$.

\begin{proposition}\label{4.1}
Let $A,B,C\in \mathbb{C}_{n,n}$. If $A$ is $(B,C)$-core-EP invertible with $k\in {\rm I}_{(B,C)}(A)$, then $\mathbb{C}^n=\mathcal{R}((CA)^{k}B)\oplus \mathcal{N}(A_{(B,C)}^{\tiny{\textcircled{D}}}CA)$.
\end{proposition}

\begin{theorem}\label{4.2}
Let $A,B,C\in \mathbb{C}_{n,n}$. If $A$ is $(B,C)$-core-EP invertible with $k\in {\rm I}_{(B,C)}(A)$, then $X=A_{(B,C)}^{\tiny\textcircled{\tiny{D}}}$ is the unique solution to the system of conditions
\begin{center} $CAX=P_{(CA)^{k}C}$ and $\mathcal{R}(X)\subseteq\mathcal{R}((CA)^{k}B)$.
\end{center}
\end{theorem}

\begin{proof} As $X=A_{(B,C)}^{\tiny\textcircled{\tiny{D}}}$ is the $(B,C)$-core-EP inverse of $A$, then $\mathcal{R}(X)\subseteq\mathcal{R}((CA)^{k}B)$, $(CAX)^{*}=CAX$ and $XCAX=X$, and thus $\mathcal{R}(CAX)\subseteq \mathcal{R}((CA)^{k+1})\subseteq \mathcal{R}((CA)^{k}C)=\mathcal{R}(CAX(CA)^{k}C)\subseteq \mathcal{R}(CAX)$. This gives $\mathcal{R}(CAX)= \mathcal{R}((CA)^{k}C)$. From $(CAX)^{2}=CAXCAX=CAX=(CAX)^*$, it follows that $CAX=P_{CAX}=P_{(CA)^{k}C}$. Hence, $X=A_{(B,C)}^{\tiny\textcircled{\tiny{D}}}$ is the solution to the system of $CAX=P_{(CA)^{k}C}$ and $\mathcal{R}(X)\subseteq\mathcal{R}((CA)^{k}B)$.

Let $X_{1}, X_{2}\in  \mathbb{C}_{n,n}$ be any two solutions to the system of $CAX=P_{((CA)^{k}C)}$ and $\mathcal{R}(X)\subseteq\mathcal{R}((CA)^{k}B)$. Then $CA(X_{1}-X_{2})=0$, whence $\mathcal{R}(X_{1}-X_{2})\subseteq\mathcal{N}(CA)$. By $\mathcal{R}(X_{1})\subseteq\mathcal{R}((CA)^{k}B)$ and $\mathcal{R}(X_{2})\subseteq\mathcal{R}((CA)^{k}B)$, one has $\mathcal{R}(X_{1}-X_{2})\subseteq\mathcal{R}((CA)^{k}B)$, so that $\mathcal{R}(X_{1}-X_{2})\subseteq\mathcal{N}(CA)\cap\mathcal{R}((CA)^{k}B)\subseteq \mathcal{N}(A_{(B,C)}^{\tiny\textcircled{\tiny{D}}}CA)\cap\mathcal{R}((CA)^{k}B)=\{0\}$ by Proposition \ref{4.1}. Therefore, $X_{1}=X_{2}$.
\hfill$\Box$
\end{proof}

Consider the linear equation $Ax=b$, where $A\in\mathbb{C}_{n, n}$ and $b\in\mathbb{C}^{n}$. If ${\rm ind}(A)=k$ (i.e., $k$ is the smallest nonnegative integer such that ${\rm rk}(A^k)={\rm rk}(A^{k+1})$ and $b\in \mathcal{R}(A^k)$, it is known that $x=A^Db$ is the unique solution to the equation $Ax=b$ with respect to $x\in \mathcal{R}(A^k)$. In particular, if ${\rm ind}(A)=1$ and $b\in \mathcal{R}(A)$, then $x=A^\#b$ ($A^\#$ stands for the group inverse of $A$) is the unique solution to $Ax=b$ with respect to $x\in \mathcal{R}(A)$.

If $b \notin \mathcal{R}(A)$, then the equation $Ax=b$ is unsolvable, but has least-squares solutions. Recently, the present author Zhu et al. \cite[Theorem 4.7]{Zhu20241} established the unique solution to the constrained matrix approximation problem in the Euclidean norm. Precisely, it is shown that $x=A_{(B,C)}^{\tiny\textcircled{\tiny{\#}}}b$ is a unique solution to the constrained matrix approximation problem \begin{center}${\rm min}\|CAx-b\|_{2}$ subject to $x\in\mathcal{R}(B)$, where $b\in \mathbb{C}^{n}$.\end{center}

Due to the nature of the direct-sum decomposition, similar result holds for the system of linear equations, whose proof works can be given by simulating the behaviour of Theorem \ref{4.2}. Herein a lemma is given.

\begin{lemma}\label{4.3} Let $A,B,C\in \mathbb{C}_{n,n}$ and $b\in \mathbb{C}^{n}$. If $A_{(B,C)}^{\tiny\textcircled{\tiny{D}}}$ exists and $k\in {\rm I}_{(B,C)}(A)$, then $x=A_{(B,C)}^{\tiny\textcircled{\tiny{D}}}b$ is the unique solution to the system of conditions
 \begin{center} $CAx=P_{(CA)^{k}C}b$ and $x\in \mathcal{R}((CA)^{k}B)$.\end{center}
\end{lemma}

With this lemma, we can now prove an important result.

\begin{theorem}\label{4.4}
Let $A,B,C\in \mathbb{C}_{n,n}$ be such that $A_{(B,C)}^{\tiny\textcircled{\tiny{D}}}$ exists with $k\in I_{(B,C)}(A)$. Then, for any given $b\in\mathbb{C}^{n}$, $x=A_{(B,C)}^{\tiny\textcircled{\tiny{D}}}b$ is the unique solution to \begin{center} ${\rm min} \|CAx-b\|_{2}$   {\rm subject to}   $x\in \mathcal{R}((CA)^{k}B)$. \end{center}
\end{theorem}

\begin{proof}
Write $CAx-b=(CAx-CAA_{(B,C)}^{\tiny\textcircled{\tiny{D}}}b)+(CAA_{(B,C)}^{\tiny\textcircled{\tiny{D}}}b-b)=x_1+x_2$, where $x_1=CAx-CAA_{(B,C)}^{\tiny\textcircled{\tiny{D}}}b$ and $x_2=CAA_{(B,C)}^{\tiny\textcircled{\tiny{D}}}b-b$. Note that $x\in \mathcal{R}((CA)^{k}B)$. Then there is some $y\in \mathbb{C}^{n}$ such that $x=(CA)^{k}By$, so that $CAA_{(B,C)}^{\tiny\textcircled{\tiny{D}}}CAx=CAA_{(B,C)}^{\tiny\textcircled{\tiny{D}}}(CA)^{k+1}By=(CA)^{k+1}By=CAx$. We claim that $x_1^*x_2=0$. Indeed, we have \begin{eqnarray*}x_1^*x_2&=&(CAx-CAA_{(B,C)}^{\tiny\textcircled{\tiny{D}}}b)^*(CAA_{(B,C)}^{\tiny\textcircled{\tiny{D}}}b-b)\\
&=&((CAx)^*-b^*CAA_{(B,C)}^{\tiny\textcircled{\tiny{D}}})(CAA_{(B,C)}^{\tiny\textcircled{\tiny{D}}}b-b)\\
&=&(CAx)^*CAA_{(B,C)}^{\tiny\textcircled{\tiny{D}}}b-b^*(CAA_{(B,C)}^{\tiny\textcircled{\tiny{D}}})^2b-(CAx)^*b+ b^*CAA_{(B,C)}^{\tiny\textcircled{\tiny{D}}}b \\
&=&(CAx)^*CAA_{(B,C)}^{\tiny\textcircled{\tiny{D}}}b-(CAx)^*b\\
&=&(CAx)^*(CAA_{(B,C)}^{\tiny\textcircled{\tiny{D}}})^*b-(CAx)^*b\\
&=&(CAA_{(B,C)}^{\tiny\textcircled{\tiny{D}}}CAx)^*b-(CAx)^*b\\
&=&(CAx)^*b-(CAx)^*b\\
&=&0.
\end{eqnarray*}

Notice also that $x_1^*x_2=0$ implies $x_2^*x_1=0$. Then $\|CAx-b\|_{2}^{2}=\|CAx-CAA_{(B,C)}^{\tiny\textcircled{\tiny{D}}}b\|_{2}^{2}+\|CAA_{(B,C)}^{\tiny\textcircled{\tiny{D}}}b-b\|_{2}^{2}$, which implies that $\|CAx-b|_{2}^{2}$ is minimal if $\|CAx-CAA_{(B,C)}^{\tiny\textcircled{\tiny{D}}}b\|_{2}^{2}=0$, i.e., $CAx=CAA_{(B,C)}^{\tiny\textcircled{\tiny{D}}}b=P_{(CA)^{k}C}b$. Since $x\in \mathcal{R}((CA)^{k}B)$, it follows from Corollary \ref{4.3} that $x=A_{(B,C)}^{\tiny\textcircled{\tiny{D}}}b$ is the unique solution to ${\rm min} \|CAx-b\|_{2}$ {\rm subject to}   $x\in \mathcal{R}((CA)^{k}B)$. \hfill$\Box$
\end{proof}

\begin{corollary}
Let $A,B,C\in \mathbb{C}_{n,n}$ be such that $A_{(B,C)}^{\tiny\textcircled{\tiny{\#}}}$ exists. Then, for any given $b\in\mathbb{C}^{n}$, $x=A_{(B,C)}^{\tiny\textcircled{\tiny{\#}}}b$ is the unique solution to \begin{center} ${\rm min} \|CAx-b\|_{2}$   {\rm subject to}  $x\in \mathcal{R}(B)$. \end{center}
\end{corollary}

\section{Discussion}

The paper defines the $(b,c)$-core-EP inverse of $a$, and exhibits its properties, connections with other types of generalized inverses and applications in a $*$-monoid $S$. It should be noted that one also can define and investigate the dual $(b,c)$-core-EP inverse as a dual notion to the current $(b,c)$-core-EP inverse as follows. For any $a,b,c\in S$, $a$ is called dual $(b,c)$-core-EP invertible if there are some $y\in S$ and $k\in \mathbb{N}$ such that $(ca)^kb=(ca)^kbyab$, $yS=((ca)^kb)^*S$ and $Sy=S(ca)^kc$. So, all results established in this paper for the $(b,c)$-core-EP inverse have dual versions for the dual $(b,c)$-core-EP inverse.

\bigskip
\centerline {\bf Disclosure statement}
\vskip 2mm
No potential conflict of interest was reported by the authors.

\bigskip
\centerline {\bf ACKNOWLEDGMENTS}
\vskip 2mm
The authors sincerely thank Professor Alberto Facchini, University of Padova, Italy for his comments on an earlier version of this manuscript. This research is supported by the National Natural Science Foundation of China (No.
11801124).

\begin{flushleft}
{\bf References}
\end{flushleft}


\begin{thebibliography}{00}

\bibitem{Anderson1974} F.W. Anderson, K.R. Fuller, Rings and categories of modules, Springer, Berlin and New York, 1974

\bibitem{Baksalary2010} O.M. Baksalary, G. Trenkler, Core inverse of matrices, Linear Multilinear Algebra 58 (2010) 681-697.

\bibitem{Ben1976} A. Ben-Israel, Applications of generalized inverses to programming, games and networks, Generalized Inverses and Applications, Academic Press, (1976) 495-523.

\bibitem{Ben-Israel2003} A. Ben-Israel, T.N.E. Greville, Generalized inverses: Theory and Applications, 2nd ed. Springer, New York, 2003.

\bibitem{BD1953} R. Bott, R.J. Duffin, On the algebra of networks, Trans. Amer. Math. Soc. 74 (1953) 99-109.

\bibitem{Brookes2024} M. Brookes, C. Miller, Heights of posets associated with Green's relations on semigroups, J. Algebra, 659 (2024) 109-131.

\bibitem{Dawson1997} E. Dawson, C.K. Wu, Key agreement scheme based on generalised inverses of matrices, Electron Lett 33 (1997) 1210-1211.

\bibitem{Dolinar2019} G. Dolinar, B. Kuzma, J. Marovt, B. Ungor, Properties of core-EP order in rings with involution, Front. Math. China 14 (2019) 715-736.

\bibitem{Drazin2012} M.P. Drazin, A class of outer generalized inverses, Linear Algebra Appl. 436 (2012) 1909-1923.

\bibitem{Drazin2021} M.P. Drazin, Hybrid $(b,c)$-inverses and five finiteness properties in rings, semigroups and categories, Comm. Algebra 49 (2021) 2265-2277.

\bibitem{Drazin2016} M.P. Drazin, Left and right generalized inverses, Linear Algebra Appl. 510 (2016) 64-78.

\bibitem{Drazin1958} M.P. Drazin, Pseudo-inverses in associative rings and semigroups, Amer. Math. Monthly 65 (1958) 506-514.

\bibitem{Elliott2006} R.J. Elliott, J. Van der Hoek, Optimal linear estimation and data fusion, IEEE Trans. Automat. Control 51 (2006) 686-689.

\bibitem{Fredholm1903} I. Fredholm, Sur une classe d'\'{e}quations fonctionnelles, Acta Math. 27 (1903) 365-390.

\bibitem{Gao2013} Y.F Gao, J.L Chen, Pseudo core inverses in rings with involution, Comm. Algebra 61 (2013) 886-891.

\bibitem{Green1951} J.A. Green, On the structure of semigroups, Ann. Math. 54 (1951) 163-172.

\bibitem{Howie1955} J.M. Howie, Fundamentals of semigroup theory, volume 12 of London Mathematical Society Monographs. New Series. The Clarendon Press, Oxford University Press, New York, 1995. Oxford Science Publications.

\bibitem{Hunter2014} J.J. Hunter, Generalized inverses of Markovian kernels in terms of properties of the Markov chain, Linear Algebra Appl. 447 (2014) 38-55.

\bibitem{Koliha1996} J.J. Koliha, A generalized Drazin inverse, Glasgow Math. J. 38 (1996) 367-381.

\bibitem{Kyrchei2021} I. Kyrchei, D. Mos\'{i}c, P.S. Stanimirov\'{i}c, Solvability of new constrained quaternion matrix approximation problems based on Core-EP Inverses, Adv. Appl. Clifford Algebras 31 (2021), DOI:/10.1007/s00006-020-01102-7.

\bibitem{Lam2014} T.Y. Lam, Pace P. Nielsen, Inner inverses and inner annihilators in rings, J. Algebra 397 (2014) 91-110.

\bibitem{Lam2019} T.Y. Lam, Pace P. Nielsen, Jacobson pairs and Bott-Duffin decompositions in rings, Contemp. Math. 727 (2019) 249-267.

\bibitem{Ma2018} H.F. Ma, P.S. Stanimirov\'{i}c, Characterizations, approximation and perturbations of the core-EP inverse, Appl. Math. Comput. 359 (2019) 404-417.

\bibitem{Malik2014} S.B. Malik, N. Thome, On a new generalized inverse for matrices of an arbitrary index, Appl. Math. Comput. 226 (2014) 575-580.

\bibitem{Prasad1992} K. Manjunatha Prasad, R.B. Bapat, The generalized Moore-Penrose inverse, Linear Algebra Appl. 165 (1992) 59-69.

\bibitem{Prasad2014} K. Manjunatha Prasad, K.S. Mohana, Core-EP inverses, Linear Multilinear Algebra 62 (2014) 792-802.

\bibitem{Marquaridt1970} D.W. Marquaridt, Generalized inverses, ridge regression, biased linear estimation, and nonlinear estimation, Technometrics 12 (1970) 591-612.

\bibitem{Mary2011} X. Mary, On generalized inverses and Green's relations, Linear Algebra Appl. 434 (2011) 1836-1844.

\bibitem{Mary2023} X. Mary, On the structure of semigroups whose regular elements are completely regular, Semigroup Forum 107 (2023) 692-717.

\bibitem{Meyer1975} C.D. Meyer, The role of the group generalised inverse in the theory of finite Markov chains, SIAM Rev. 17 (1975) 443-464.

\bibitem{Moore1920} E.H. Moore, On the reciprocal of the general algebraic matrix, Bull. Amer. Math. Soc., 26(1920) 394-395.

\bibitem{Mosic2020} D. Mosi\'{c}, Core-EP inverse in rings with involution. Publ. Math. Debrecen 96 (2020) 427-443.

\bibitem{Mosic20201} D. Mosi\'{c}, P.S. Stanimirovi\'{c}, V.N. Katsikis, Solvability of some constrained matrix approximation problems using core-EP inverses, Comput. Appl. Math. 39 (2020) 311-333.

\bibitem{Mosic2023} D. Mosi\'{c}, H.H. Zhu, L.Y. Wu, The $w$-core-EP inverse in rings with involution, Rev. de la Union Mat. Argentina, (2023),  DOI:10.33044/revuma.3478.

\bibitem{Penrose1955} R. Penrose, A generalized inverse for matrices, Proc. Camb. Phil. Soc. 51 (1955) 406-413.

\bibitem{Rakic2014} D.S. Raki\'{c}, N.C. Din\v{c}i\'{c}, D.S. Djordjevi\'{c}, Group, Moore-Penrose, core and dual core inverse in rings with involution, Linear Algebra Appl. 463 (2014) 115-133.

\bibitem{Sun2001} H.M. Sun, Cryptanalysis of a public-key cryptosystem based on generalized inverses of matrices, IEEE Commun Lett 5 (2001) 61-63.

\bibitem{Wu1998} C.K. Wu, E. Dawson, Generalised inverses in public key cryptosystem design, IEE P-Compit Dig T, 145 (1998) 321-326.

\bibitem{Zhu2024} H.H. Zhu, The $(b,c)$-core inverse and its dual in semigroups with involution, J. Pure Appl. Algebra 228 (2024) 107526.

\bibitem{Zhu2016} H.H. Zhu, J.L. Chen, P. Patr\'{i}cio, Further results on the inverse along an element in semigroups and rings,
Linear Multilinear Algebra 64 (2016) 393-403.

\bibitem{Zhu20231} H.H. Zhu, C.C. Wang, Q.W. Wang, Left $w$-core inverses in rings with involution, Mediterr. J. Math. 20 (2023) 337-356.

\bibitem{Zhu2023} H.H. Zhu, L.Y. Wu, J.L. Chen, A new class of generalized inverses in semigroups and rings with involution, Commun. Algebra 71 (2023) 528-544.

\bibitem{Zhu20232} H.H. Zhu, L.Y. Wu, D. Mos\'{i}c, One-sided $w$-core inverses in rings with an involution, Linear Multilinear Algebra 71 (2023) 528-544.

\bibitem{Zhu20241} H.H. Zhu, D.H. Xu, P. Patr\'{i}cio, D. Mosi\'{c}, The $(B, C)$-core inverse of rectangular matrices, Quaest. Math. 47 (2024) 1-18.

\bibitem{Zhu2015} H.H. Zhu, X.X. Zhang, J.L. Chen, Generalized inverses of a factorization in a ring with involution,  Linear Algebra Appl. 472 (2015) 142-150.

\bibitem{Zhu2019} H.H. Zhu, H.L. Zou, P. Patricio, Generalized inverses and their relations with clean decompositions, J. Algebra Appl. 18  (2019) 1950133.
\end{thebibliography}
\end{document}